\documentclass[final]{amsart}
\usepackage{amssymb}
\usepackage{hhline}
\usepackage{graphicx}

\newtheorem{theorem}{Theorem}[section]
\newtheorem{lemma}[theorem]{Lemma}
\newtheorem{proposition}[theorem]{Proposition}

\theoremstyle{definition}

\newcommand{\Aut}{\mathrm{Aut\mkern 2mu}}
\newcommand{\M}{\mathrm{M\mkern 1mu}}

\title{Interassociativity and three-element doppelsemigroups}
\author{Volodymyr Gavrylkiv and Diana Rendziak}
\address[V.~Gavrylkiv]{Vasyl Stefanyk Precarpathian National University, Ivano-Frankivsk, Ukraine} \email{vgavrylkiv@gmail.com}
\address[D.~Rendziak]{Vasyl Stefanyk Precarpathian National University, Ivano-Frankivsk, Ukraine} \email{rendziakdiana007@gmail.com}

\begin{document}

\begin{abstract}
In the paper  we characterize all interassociates of  some
non-inverse semigroups and describe up to isomorphism all
three-element (strong) doppelsemigroups and their automorphism
groups. We prove that there exist $75$ pairwise non-isomorphic three-element
doppelsemigroups among which $41$ doppelsemigroups are
commutative. Non-commutative doppelsemigroups are divided into
$17$ pairs of dual doppelsemigroups. Also up to isomorphism there
are $65$ strong doppelsemigroups of order $3$, and all non-strong
doppelsemigroups are not commutative. 
\end{abstract}

\maketitle

\section*{Introduction}

Given a semigroup $(S, \dashv)$, consider a semigroup $(S, \vdash)$ defined on the same set. We say that $(S, \vdash)$ is an {\em interassociate} of $(S, \dashv)$ provided
$(x\dashv y)\vdash z=x\dashv(y\vdash z)$ and $(x\vdash y)\dashv z=x\vdash(y\dashv z)$
for all $x, y, z\in S$. In 1971, Zupnik~\cite{Zu} coined the term interassociativity in
a general groupoid setting. However, he required only one of the two defining
equations to hold. The present concept of interassociativity for semigroups
originated in 1986 in Drouzy~\cite{Dr}, where it is noted that every group is isomorphic
to each of its interassociates. In 1983, Gould and Richardson~\cite{GR} introduced
{\em strong interassociativity}, defined by the above equations along with  $x\dashv (y\vdash z)=x\vdash(y\dashv z)$. J. B. Hickey in 1983~\cite{Hi83} and 1986~\cite{Hi86} dealt with the special case of interassociativity in which the operation $\vdash$ is defined by specifying $a\in S$ and stipulating that $x \vdash y = x\dashv a\dashv y$ for all $x, y\in S$. Clearly $(S, \vdash)$, which Hickey calls a {\em variant} of $(S, \dashv)$, is a semigroup that is an interassociate of $(S, \dashv)$. It is  easy to show that if $(S, \dashv)$ is a monoid, every interassociate $(S, \vdash)$ must satisfy the condition $x\vdash y=x\dashv a\dashv y$ for some fixed element $a\in S$ and for all $x,y\in S$, that is $(S, \vdash)$ is a variant of $(S, \dashv)$. Methods of constructing interassociates were developed, for semigroups in general and for specific classes of semigroups, in 1997 by Boyd, Gould and Nelson~\cite{BoGN}. The description of all interassociates of finite
monogenic semigroups was presented by Gould, Linton and Nelson in 2004, see~\cite{GLN}.

A {\em doppelsemigroup} is an algebraic structure $(D,\dashv,\vdash)$ consisting of a non-empty set $D$ equipped with two associative binary operations $\dashv$ and $\vdash$ satisfying the following axioms:
$$(D_1)\ \ (x\dashv y)\vdash z=x\dashv(y\vdash z),$$
$$(D_2)\ \ (x\vdash y)\dashv z=x\vdash(y\dashv z).$$

Thus, we can see that in any doppelsemigroup $(D,\dashv,\vdash)$, $(D,\vdash)$ is
an interassociate of $(D,\dashv)$, and conversely, if a semigroup $(D,\vdash)$ is an
interassociate of a semigroup $(D,\dashv)$ then $(D,\dashv,\vdash)$ is a doppelsemigroup.
A doppelsemigroup $(D,\dashv,\vdash)$ is called {\em commutative}~\cite{Zh2017AU} if both semigroups $(D,\dashv)$ and $(D,\vdash)$ are commutative.
A doppelsemigroup $(D,\dashv,\vdash)$ is said to be {\em strong}~\cite{Zh2018} if it satisfies the axiom
$x\dashv (y\vdash z)=x\vdash(y\dashv z)$.

Many classes of doppelsemigroups were studied by Anatolii~Zhuchok and Yurii~Zhuchok. The free product of doppelsemigroups, the free (strong) doppelsemigroup, the free commutative (strong) doppelsemigroup, the free n-nilpotent (strong) doppelsemigroup and the free rectangular doppelsemigroup were constructed in~\cite{Zh2017AU, Zh2018, ZhZhK2020}. Relatively free doppelsemigroups were studied in~\cite{Zh2018M}.  The free n-dinilpotent (strong) doppelsemigroup was constructed in~\cite{ZhD2016, Zh2018}. In~\cite{Zh2017} A.~Zhuchok described the free left n-dinilpotent doppelsemigroup. Representations of ordered doppelsemigroups by binary relations were studied   by Y.~Zhuchok and  J.~Koppitz, see~\cite{ZhK2019}. 

Until now, the task of describing all pairwise non-isomorphic
(strong) doppelsemigroups of order $3$ has not been solved. The goal of the present work is to characterize all interassociates   of some
non-inverse semigroups, and use these characterizations in
describing up to isomorphism all three-element (strong)
doppelsemigroups and their automorphism groups.

\section{Preliminaries}

A semigroup $S$ is called an {\em inflation} of its subsemigroup $T$ (see~\cite{CP}, Section 3.2)  provided that there is an surjective map $r: S\to T$ such that $r^2=r$ and $r(a)r(b)=ab$ for all $a,b\in S$. In the described situation $S$ is often referred to as an {\em inflation of $T$ with an associated map $r$} (or just {\em with a map $r$)}. It is immediate that if $S$ is an inflation of $T$ then $T$ is a retract of $S$ (that is the image under a {\em retraction} $r$ in the sense that $r(a)=a$ for all $a\in T$) and $S^2\subset T$.

A semigroup $S$ is called {\em monogenic} if it is generated by some element $a\in S$ in the sense that $S=\{a^n\}_{n\in\mathbb N}$. If a monogenic semigroup is
infinite then it is isomorphic to the additive semigroup $\mathbb N$ of positive integer numbers. A
finite monogenic semigroup $S=\langle a\rangle$ also has simple
structure, see  \cite{Howie}. There are positive integer numbers $r$
and $m$ called the {\em index} and the {\em period} of $S$ such
that
\begin{itemize}
    \item $S=\{a,a^2,\dots,a^{r+m-1}\}$ and $r+m-1=|S|$;
    \item $a^{r+m}=a^{r}$;
    \item $C_m:=\{a^r,a^{r+1},\dots,a^{r+m-1}\}$ is
    a cyclic and maximal subgroup of $S$ with the
    neutral element $e=a^n\in C_m$ and generator $a^{n+1}$, where $n\in (m\cdot\mathbb N)\cap\{r,\dots,r+m-1\}$.
\end{itemize}

We denote by $\M_{r,m}$ a finite monogenic semigroup of index $r$ and period $m$.

\smallskip

Recall that an {\em isomorphism} between $(S,*)$ and $(S',\circ)$ is a bijective
function $\psi:S\to S'$ such that $\psi(x*y)=\psi(x)\circ \psi(y)$ for
all $x,y\in S$. If there exists an isomorphism between $(S,*)$ and
$(S',\circ)$ then $(S,*)$ and $(S',\circ)$ are said to be {\em isomorphic}, denoted
$(S,*)\cong (S',\circ)$. An isomorphism between $(S,*)$ and $(S,*)$ is called an {\em
automorphism} of a semigroup $(S,*)$. By $\Aut (S,*)$ we denote the
automorphism group of a semigroup $(S,*)$.

An element $e$ of a semigroup $(S,*)$ is called an {\em idempotent} if $e*e=e$.  The semigroup is a {\em band}, if all its elements are idempotents. Commutative bands are called {\em semilattices}.
By $L_n$ we denote the {\em linear semilattice} $\{0, 1,\ldots, n-1\}$ of order $n$, endowed with the operation of minimum.

If $(S,*) $ is a semigroup then the semigroup $(S,{*}^d)$ with operation $x{*}^d y=y* x$ is
called {\em dual} to $(S,*)$.

\smallskip

A non-empty subset $I$ of a semigroup $(S,*)$ is called an {\em ideal}
if $I*S\cup S*I\subset I$. An element $z$ of a semigroup $S$ is called a {\em zero} (resp. a
{\em  left zero}, a {\em right zero}) in $S$ if $a*z=z*a=z$ (resp. $z*a=z$, $a*z=z$) for any $a\in S$. If $(D,\dashv, \vdash)$ is a doppelsemigroup and $z\in D$ is a { zero} (resp. a
{ left zero}, a { right zero}) of a semigroup $(D,\dashv)$ then $(D_1)$ and $(D_2)$ imply that $z$ is a zero (resp. a { left zero}, a { right zero}) of a semigroup $(D,\vdash)$, and vice versa. Thus, any interassociate of a semigroup with zero is a semigroup with zero as well.

A semigroup $(S,*)$ is called a {\em null semigroup} if there exists
an element $z\in S$ such that $x*y=z$ for any $x,y\in S$. In this
case  $z$ is a zero of $S$. All null semigroups on the same set
are isomorphic. By $O_X$ we denote a null semigroup on a set $X$.
If $X$ is finite of cardinality $|X|=n$ then instead of $O_X$ we
use $O_n$.  It is easy to see that a null semigroup is a strong
interassociate of each semigroup  with the same zero.

\smallskip

Let $X$ be a set, $z\in X$ and $A\subset X\setminus\{z\}$. Define  the binary operation $*$ on  $X$ in the following way:

$$x* y=\begin{cases}
x,\ y=x\in A \\
z,\ \text{otherwise}.
\end{cases}$$

It is easy to check that a set $X$ endowed with the operation $*$ is a semigroup with  zero $z$, and we denote this semigroup by $O^A_X$.
If $A=X\setminus\{z\}$ then $O^A_X$ is a semilattice. In the case $A=\emptyset$,  $O^A_X$ coincides with a null  semigroup  with  zero $z$.
The semigroups $O^A_X$ and $O^B_Y$ are isomorphic if and only if $|X|=|Y|$ and $|A|=|B|$. If $X$ is a finite set of cardinality $|X|=n$ and $|A|=m$ then we use $O_n^m$ instead of $O^A_X$.

\smallskip

Let $(S,*)$ be a semigroup and $e\notin S$. The binary operation $*$ defined on  $S$  can be extended to $S\cup\{e\}$ putting $e*s=s*e=s$ for all $s\in S\cup \{e\}$.
The notation $(S,*)^{+1}$  denotes a monoid $(S\cup\{e\}, *)$ obtained from $(S,*)$ by adjoining the extra identity $e$ (regardless of whether $(S,*)$ is or is not a monoid).

Let $(S,*)$ be a semigroup and $0\notin S$. The binary operation $*$ defined on  $S$  can be extended to $S\cup\{0\}$ putting $0*s=s*0=0$ for all $s\in S\cup \{0\}$.
The notation $(S,*)^{+0}$  denotes a semigroup $(S\cup\{0\},*)$ obtained from $(S,*)$ by adjoining the extra zero $0$ (regardless of whether $(S,*)$ has or has not the zero).

Let $(M, *)$ be a monoid with  identity $e$, and  $\tilde{1}\notin M$.
The binary operation $*$ defined on  $M$  can be extended to $M\cup\{\tilde{1}\}$
putting ${\tilde{1}}*{\tilde{1}}=e$ and $\tilde{1}* m=m* \tilde{1}=m$ for all $m\in M$.
The notation $(M,*)^{\tilde{1}}$  denotes the semigroup  obtained from $(M,*)$ by adjoining an extra element $\tilde{1}$. Note that  $(M,*)^{\tilde{1}}$ is not a monoid and $(M,*)^{\tilde{1}}$ is an inflation of a monoid $(M,*)$.

Let $(D,\dashv, \vdash)$ be a doppelsemigroup and $0\notin D$. The binary operations defined on  $D$  can be extended to $D\cup\{0\}$ putting $0\dashv d=d\dashv 0=0=0\vdash d=d\vdash 0 $ for all $d\in D\cup \{0\}$. The notation $(D,\dashv, \vdash)^{+0}$  denotes a doppelsemigroup $(D\cup\{0\},\dashv, \vdash)$ obtained from $(D,\dashv, \vdash)$ by adjoining the extra zero $0$. If $(D,\dashv, \vdash)$ is a strong doppelsemigroup then $(D,\dashv, \vdash)^{+0}$ is a strong doppelsemigroup as well. It is easy to see that $\Aut((D,\dashv, \vdash)^{+0})\cong\Aut(D,\dashv, \vdash)$.

A semigroup $(S,*)$ is said to be a {\em left (right) zero semigroup}
if $a*b=a$ ($a*b=b$) for any $a,b\in S$. By  $LO_X$ and $RO_X$ we
denote  a left zero semigroup and a right zero semigroup on
a set $X$, respectively.  It is easy to see that the semigroups $LO_X$ and $RO_X$ are dual. If $X$ is finite of cardinality $|X|=n$ then instead of $LO_X$ and $RO_X$ we use $LO_n$ and $RO_n$, respectively.

\smallskip

Let $X$ be a set, $A\subset X$ and $0\notin X$. Define  the binary operation $*$ on  $X^0=X\cup\{0\}$ in the following way:

$$x* y=\begin{cases}
x, \text{ if }y \in A \\
0, \text{ if }y \in X^0\setminus A.
\end{cases}$$

It is easy to check that a set $X^0$ endowed with the operation $*$ is a semigroup with  zero $0$, and we denote this semigroup by $LO^{\sim 0}_{A\leftarrow X}$. If $A=X$ then $LO^{\sim 0}_{A\leftarrow X}$ coincides with $LO^{+0}_X$. In the case $A=\emptyset$,  $LO^{\sim 0}_{A\leftarrow X}$ coincides with a null semigroup $O_{X^0}$ with  zero $0$. The semigroups $LO^{\sim 0}_{A\leftarrow X}$ and $LO^{\sim 0}_{B\leftarrow Y}$ are isomorphic if and only if $|X|=|Y|$ and $|A|=|B|$. If $X$ is a finite set of cardinality $|X|=n$ and $|A|=m$ then we use $LO^{\sim 0}_{m\leftarrow n}$ instead of $LO^{\sim 0}_{A\leftarrow X}$.

\smallskip

Let $X$ be a set, $A\subset X$ and $0\notin X$. Define the binary operation $*$ on  $X^0=X\cup\{0\}$ in the following way:

$$x* y=\begin{cases}
y, \text{ if }x \in A\\
0, \text{ if }x \in X^0\setminus A.
\end{cases}$$

It is easy to check that a set $X^0$ endowed with the operation $*$ is a semigroup with  zero $0$, and we denote this semigroup by $RO^{\sim 0}_{A\leftarrow X}$. If $A=X$ then $RO^{\sim 0}_{A\leftarrow X}$ coincides with $RO^{+0}_X$. In the case $A=\emptyset$, $RO^{\sim 0}_{A\leftarrow X}$ coincides with a null semigroup on ${X^0}$ with  zero $0$. Semigroups $RO^{\sim 0}_{A\leftarrow X}$ and $RO^{\sim 0}_{B\leftarrow Y}$ are isomorphic if and only if $|X|=|Y|$ and $|A|=|B|$. If $X$ is a finite set of cardinality $|X|=n$ and $|A|=m$ then we use $RO^{\sim 0}_{m\leftarrow n}$ instead of $RO^{\sim 0}_{A\leftarrow X}$.

 It is easy to see that the semigroups $LO^{\sim 0}_{A\leftarrow X}$ and $RO^{\sim 0}_{A\leftarrow X}$ are dual.
 
\smallskip 

Let  $a$ and $c$ be different elements of a set $X$. Define the associative binary operation $\dashv_c^a$ on  $X$ in the following way:

$$x\dashv_c^a y=\begin{cases}
x,\text{ if } x\neq c \\
a,\text{ if } x=c\text{ and } y\neq c\\
c,\text{ if } x=y=c.
\end{cases}$$

It follows that $(X,\dashv_c^a)$ is a non-commutative band in
which all elements $z\neq c$ are left zeros.

It is not difficult to check that for any different $b, d\in X$,
the semigroups $(X,\dashv_c^a)$ and $(X,\dashv_d^b)$ are
isomorphic. We denote  by $LOB_X$ a model semigroup of the class
of semigroups isomorphic to $(X,\dashv_c^a)$. If $X$ is a finite
set of cardinality $|X|=n$  then we use $LOB_n$ instead of
$LOB_X$. 

 The semigroup $ROB_X$ is defined dually.

\smallskip

Let  $a$ and $c$ be different elements of a set $X$. Define the associative binary operation $\vdash^a_c$ on  $X$ in the following way:

$$x\vdash^a_c y=\begin{cases}
x,\ x\neq c \\
a,\ x=c.
\end{cases}$$

It follows that $(X,\vdash_c^a)$ is a non-commutative non-regular semigroup in which all elements $z\neq c$ are left zeros.

It is not difficult to check that for any $b\neq c$,
the semigroups $(X,\vdash_c^a)$ and $(X,\vdash_c^b)$ are
isomorphic. We denote  by $LO_{X\setminus\{c\}\leftarrow X}$  a model semigroup of the class of semigroups isomorphic to $(X,\vdash_c^a)$. If $X$ is a finite set of cardinality $|X|=n$  then we use $LO_{(n-1)\leftarrow n}$ instead of $LO_{X\setminus\{c\}\leftarrow X}$. 

Dually we  define the semigroups  $RO_{X\setminus\{c\}\leftarrow X}$ and $RO_{(n-1)\leftarrow n}$.

\smallskip

A transformation $l : S\to S$ of a semigroup $(S,*)$ is called a {\em left translation} if $l(x* y)=l(x)* y$ for all $x,y\in S$. By Corollary 2.2. from~\cite{BoGN} for any semigroup $(S,*)$ and for any its left translation $l$, the semigroup $(S,*_l)$, where $x*_ly=x* l(y)$, is an interassociate of $(S,*)$. Thus,
$(S,*,*_l)$ is a doppelsemigroup for any left translation $l : S\to S$.

The following lemma was proved in~\cite{BoGN}.

\begin{lemma}\label{inflation}
 Let $(S,*)$ be an inflation of an inverse Clifford semigroup $(A,*)$ and let
$r : S \to A$ denote the associated retraction. If $(S,\circ)$ is a semigroup that is an interassociate of $(S,*)$ then $A$ is an ideal of $(S,\circ)$ and $(A,\circ)=(A, *_l)$ for some left translation $l$ of $(A,*)$. Moreover, $r$ is a homomorphism of $(S, \circ)$ onto $(A,\circ)$.
\end{lemma}

\section{Isomorphisms of doppelsemigroups}

A bijective map $\psi : D_1 \to D_2$ is called an {\em isomorphism of doppelsemigroups}  $(D_1,\dashv_1, \vdash_1)$ and $(D_2,\dashv_2, \vdash_2)$ if $$\psi(a\dashv_1 b)=\psi(a)\dashv_2\psi(b)\ \ \text{  and  }\ \ \psi(a\vdash_1 b)=\psi(a)\vdash_2\psi(b)$$ for all $a,b\in D_1$.

If there exists an isomorphism between the doppelsemigroups $(D_1,\dashv_1, \vdash_1)$ and $(D_2, \dashv_2, \vdash_2)$ then $(D_1, \dashv_1, \vdash_1)$ and $(D_2, \dashv_2, \vdash_2)$ are said to be {\em isomorphic}, denoted
$(D_1,\dashv_1, \vdash_1)\cong (D_2,\dashv_2, \vdash_2)$. An isomorphism $\psi: D\to D$ is called an {\em   automorphism} of a doppelsemigroup $(D,\dashv, \vdash)$. By $\Aut(D,\dashv, \vdash)$ we denote the automorphism group of a doppelsemigroup $(D,\dashv, \vdash)$.

\begin{proposition}\label{nulldop}
Let $(D_1,\dashv_1, \vdash_1)$ and $(D_2, \dashv_2, \vdash_2)$ be doppelsemigroups such that   $(D_1, \dashv_1)$ and $(D_2, \dashv_2)$ are null semigroups. If the semigroups $(D_1, \vdash_1)$ and $(D_2, \vdash_2)$ are isomorphic then  the doppelsemigroups $(D_1,\dashv_1, \vdash_1)$ and $(D_2,\dashv_2, \vdash_2)$ are isomorphic as well.
\end{proposition}

\begin{proof}
Let $z_1$ and $z_2$ be  zeros of null semigroups $(D_1, \dashv_1)$ and $(D_2, \dashv_2)$, respectively. Then $z_1$ and $z_2$ are  zeros of  the semigroups $(D_1, \vdash_1)$ and $(D_2, \vdash_2)$, respectively. Let $\psi : D_1 \to D_2$ is an isomorphism of the semigroups $(D_1, \vdash_1)$ and $(D_2, \vdash_2)$. Since zeros are preserved by isomorphisms of semigroups, $\psi(z_1)=z_2$. Taking into account that $|D_1|=|D_2|$ and any map between two null semigroups of the same order that preserves  zeros is a isomorphism of these semigroups, we conclude that $\psi : D_1 \to D_2$ is an isomorphism of the doppelsemigroups $(D_1,\dashv_1, \vdash_1)$ and $(D_2,\dashv_2, \vdash_2)$.
\end{proof}

\begin{proposition}\label{uniquedop}
Let $(D_1,\dashv_1, \vdash_1)$ and $(D_2, \dashv_2, \vdash_2)$ be doppelsemigroups, and $(D_1, \vdash)\cong (D_1, \vdash_1)$ implies $\vdash\ =\ \vdash_1$ for any interassociate $(D_1, \vdash)$ of $(D_1, \dashv_1)$. If $(D_2, \dashv_2)\cong (D_1, \dashv_1)$ and $(D_2, \vdash_2)\cong (D_1, \vdash_1)$ then 
$(D_2, \dashv_2, \vdash_2)\cong$ $(D_1,\dashv_1, \vdash_1)$.
\end{proposition}

\begin{proof} Let $\psi : D_2 \to D_1$ be an isomorphism of  semigroups $(D_2, \dashv_2)$ and $(D_1, \dashv_1)$. For any $a,b\in D_1$ define the operation $\vdash_{\psi}$ on $D_1$ in the following way:
$$a\vdash_{\psi}b=\psi(\psi^{-1}(a)\vdash_2 \psi^{-1}(b)).$$
It follows that $\psi : D_2 \to D_1$ is an isomorphism from $(D_2, \dashv_2, \vdash_2)$ to $(D_1, \dashv_1, \vdash_{\psi})$, and thus $(D_1, \dashv_1, \vdash_{\psi})$ is a doppelsemigroup as an isomorphic image of the doppelsemigroup $(D_2, \dashv_2, \vdash_2)$. Taking into account that $(D_1, \vdash_{\psi})$  is an interassociate of $(D_1, \dashv_1)$ and $(D_1, \vdash_{\psi})\cong (D_2, \vdash_2)\cong (D_1, \vdash_1)$, we conclude that $\vdash_{\psi}\ =\ \vdash_1$. Therefore, $(D_2, \dashv_2, \vdash_2)\cong (D_1,\dashv_1, \vdash_1)$.
\end{proof}

\begin{proposition}\label{autnulldop}
If $(D,\dashv, \vdash)$  is a doppelsemigroup  such that  $(D,\dashv)$ is a null semigroup then $\Aut(D,\dashv, \vdash)=\Aut(D, \vdash)$.

\end{proposition}

\begin{proof} Let $z$  be a zero of a null semigroup $(D,\dashv)$. Then
$z$ is a zero of  $(D,\vdash)$. If $\psi : D \to D$ is an
automorphism of  $(D, \vdash)$ then $\psi(z)=z$. Using the
similar arguments as in the proof of Proposition~\ref{nulldop}, we
conclude that $\psi : D \to D$ is an automorphism of  $(D,
\dashv)$. It follows that $\psi\in\Aut(D,\dashv, \vdash)$.
Therefore, $\Aut(D,\dashv, \vdash)=\Aut(D, \vdash)$.
\end{proof}

Using the fact that all bijections of a left (right) zero semigroup are its automorphisms and the similar arguments as in the proof of Proposition~\ref{autnulldop}, one can prove the following proposition.

\begin{proposition}\label{autLZdop}
If $(D,\dashv, \vdash)$  is a doppelsemigroup  such that the semigroup $(D,\dashv)$ is isomorphic to $LO_X^{+0}$ or $RO_X^{+0}$ then $\Aut(D,\dashv, \vdash)=\Aut(D, \vdash)$.
\end{proposition}

\section{Interassociates of some non-inverse semigroups}

In this section we characterize all interassociates of some non-inverse semigroups which we shall use  in section~\ref{3eldsg} for describing all three-element (strong) (commutative) doppelsemigroups up to isomorphism.

In the following Propositions~\ref{intnull+0} and~\ref{intnullA} we use Lemma~\ref{inflation} to recognize all (strong) interassociates of the semigroups $O_X^{+0}$ and $O_X^A$.

Given a semigroup $(S, \cdot)$, let $Int(S, \cdot)$ denote the set of all semigroups that
are interassociates of $(S, \cdot)$.


\begin{proposition}\label{intnull+0}
Let $O_X^{+0}$ be a semigroup obtained from a null semigroup  $O_X=(X, \dashv)$ with   zero $z$ by adjoining an extra zero $0\notin X$. The set $Int(O_X^{+0})$ consists of a null semigroup $O_{X\cup\{0\}}$ with zero $0$ and  semigroups $(X, \vdash)^{+0}$ for all semigroups $(X, \vdash)$ with  zero $z$.
All interassociates of $O_X^{+0}$ are strong.
\end{proposition}

\begin{proof} The semigroup $O_X^{+0}$ is an inflation of its subsemilattice $A=\{0,z\}$
with the associated retraction $r : O_X^{+0} \to A$,
$$r(x)=\begin{cases}
0, \ x=0 \\
z, \ x\in X.
\end{cases}$$

Let $l : A \to A$ be a left translation of the semilattice $(A,
\dashv)$. Then $l(0)=l(0\dashv 0)=l(0)\dashv 0=0$. So, there are
two left translations of $A$: $l_1(x)=0$ and $l_2(x)=x$ for all
$x\in A$. Let $(X^0, \vdash)$ be any interassociate of $O_X^{+0}$,
where $X^0=X\cup\{0\}$. By Lemma~\ref{inflation}, $A$ is an ideal
of $(X^0, \vdash)$, $r$ is a homomorphism from $(X^0, \vdash)$
onto $(A, \vdash)$, and the semigroup $(A, \vdash)$ is equal to
$(A, \dashv_{l_1})$, where $x\vdash y = x\dashv_{l_1} y= x\dashv
l_1(y)=0$ for all $x,y\in A$, or  $(A, \vdash)$ is equal to $(A,
\dashv_{l_2})$, where $x\vdash y = x\dashv_{l_2} y= x\dashv
l_2(y)= x\dashv y$ for all $x,y\in A$. It follows that $(A,
\vdash)$ is a null semigroup with  zero $0$ or $(A, \vdash)=(A,
\dashv)$.

If $(A, \vdash)$ is a null semigroup then $r(x\vdash
y)=r(x)\vdash r(y)=0$. Therefore, the definition of $r$ implies
$x\vdash y=0$ for all $x,y\in X^0$. Consequently, in this case
$(X^0, \vdash)$ is a null semigroup with  zero $0$.

Let $(A, \vdash)=(A, \dashv)$. Taking into account that $r$ is a homomorphic retraction and $A\ni 0$ is an ideal, we conclude that $0\vdash x=r(0\vdash x)=r(0)\vdash r(x)=0\dashv r(x) =0$ and $x\vdash 0=r(x\vdash 0)=r(x)\vdash r(0)=r(x)\dashv 0 =0$ for all $x \in X^0$. If $x, y\in X$ then $r(x\vdash y)=r(x)\vdash r(y)=z\vdash z =z\dashv z= z$.
Thus, the definition of $r$ implies $x\vdash y\in X$ for all $x,y\in X$. Consequently, $(X, \vdash)$ is an interassociate of a null semigroup $(X, \dashv)$ with  zero $z$. It follows that $(X, \vdash)$ is an arbitrary semigroup with  zero $z$, and $(X^0, \vdash)=(X, \vdash)^{+0}$.

To show that all interassociates of $O_X^{+0}$ are strong, it is sufficient to use the following two facts:
\begin{itemize}
\item if $(X, \vdash)$ is a strong interassociate of $(X, \dashv)$ then $(X, \vdash)^{+0}$ is a strong interassociate of $(X, \dashv)^{+0}$;

\item all interassociates of a null semigroup  are strong.
\end{itemize}
\end{proof}

\begin{proposition}\label{intnullA}

A semigroup $(X, \vdash)$ with zero $z$ is an interassociate of $O_X^A$ together with the operation $\dashv$ if and only if the following conditions hold:
\begin{itemize}
    \item[1)] $(A^0, \vdash)$ coincides with $O^B_{A^0}$ for some $B\subset A$, where $A^0=A\cup\{z\}$;
    \item[2)] $A\vdash(X\setminus A)=(X\setminus A)\vdash A=\{z\}$;
    \item[3)] $X\setminus A$  is a subsemigroup with  zero $z$ of $(X, \vdash)$.
\end{itemize}
All interassociates of $O_X^A$ are strong.
\end{proposition}

\begin{proof}
Note that $O_X^A$ is an inflation of its subsemilattice $A^0$
with the associated retraction $r : O_X^A \to A^0$,
$$r(x)=\begin{cases}
x,\ x\in A \\
z,\ x\notin A.
\end{cases}$$

Let $l : A^0 \to A^0$ be a left translation of  $(A^0, \dashv)$. Then $l(z)=l(z\dashv z)=l(z)\dashv z=z$.
If $a\in A$, $l(a)=b\in A^0$ and $b\neq a$ then
$a\dashv b= z$. Therefore,  $b=b\dashv b = l(a)\dashv b=l(a\dashv b)=l(z)=z$.  It follows that $l(a)\in\{z, a\}$ for any $a\in A$.
On the other hand, it is clear that for any $B\subset A$ the map $l_B : A^0 \to A^0$,
$$l_B(x)=\begin{cases}
x,\ x\in B \\
z,\ x\in A^0\setminus B
\end{cases}$$
is a left translation of $A^0$.

Let $(X, \vdash)$ be any interassociate of $O_X^A$. By Lemma~\ref{inflation}, $A^0$ is an ideal of $(X, \vdash)$, $r$ is a homomorphism from $(X, \vdash)$ onto $(A^0, \vdash)$, and  $(A^0, \vdash)$ is equal to $(A^0, \dashv_{l_B})$, where
$$x\vdash y = x\dashv_{l_B} y= x\dashv l_B(y)=\begin{cases}
x,\ x=y\in B \\
z,\ \text{otherwise}
\end{cases}$$
for all $x,y\in A^0$. This implies $(A^0, \vdash)$ coincides with $O^B_{A^0}$ for $B\subset A$.

Since $A^0$ is an ideal of  $(X, \vdash)$, $a\vdash x,\  x\vdash
a\in A^0$ for all $a\in A^0, x\in X\setminus A$. Taking into
account that $r(a\vdash x)=r(a)\vdash r(x)=r(a)\vdash z=z$ and
$r(x\vdash a)=r(x)\vdash r(a)=z\vdash r(a)=z$ for all $a\in A^0,
x\in X\setminus A$, we conclude that $x\vdash a, a\vdash x\in
(X\setminus A)\cap A^0=\{z\}$. Therefore, $A\vdash(X\setminus
A)=(X\setminus A)\vdash A=\{z\}$.

Let us show that $X\setminus A$ is a subsemigroup of  $(X, \vdash)$. Indeed, since $r$ is a homomorphism, $r((X\setminus A)\vdash(X\setminus A))=r(X\setminus A)\vdash r(X\setminus A)=\{z\}\vdash\{z\}=\{z\}$, and the definition of $r$ implies
$(X\setminus A)\vdash(X\setminus A)\subset X\setminus A$.

Since $(X\setminus A, \dashv)$ is a null semigroup with  zero $z$, $(X\setminus A, \vdash)$ is any semigroup with the same zero $z$.

\smallskip

To show that a semigroup $(X, \vdash)$ for which the conditions 1)-3) hold is a strong interassociate of $O_X^A$, it is sufficient to note the following two facts:
\begin{itemize}

\item an element $s\in\{x\dashv(y\vdash z)$, $(x\dashv y)\vdash z$, $x\vdash(y\dashv z)$,
$(x\vdash y)\dashv z\}$ is non-zero if and only if $x=y=z\in  B$ for some $B\subset A$;

\item $b\dashv(b\vdash b)=(b\dashv b)\vdash b=b\vdash(b\dashv b)=(b\vdash b)\dashv b=b$ for any $b\in B$
for some $B\subset A$.
\end{itemize}

\end{proof}

In the following Proposition~\ref{intmonoid+almost1} we recognize all interassociates of the semigroup $(M,\dashv)^{\tilde{1}}$ for any monoid $(M,\dashv)$.

\begin{proposition}\label{intmonoid+almost1} Let $(M,\dashv)$ be a monoid with  identity $e$, and $M^{\tilde{1}}=M\cup\{\tilde{1}\}$, where $\tilde{1}\notin M$. If $(M^{\tilde{1}}, \vdash)$ is an interassociate of  $(M,\dashv)^{\tilde{1}}$ then $(M^{\tilde{1}}, \vdash)=(M,\dashv)^{+1}$ or $(M^{\tilde{1}}, \vdash)$ is a variant of $(M,\dashv)^{\tilde{1}}$ with the sandwich operation $x\vdash y=x\dashv a\dashv y$, where $a=\tilde{1}\vdash \tilde{1}\in M$. If $(M,\dashv)$ is a commutative monoid then all interassociates of $(M,\dashv)^{\tilde{1}}$ are strong interassociate with each other.
\end{proposition}

\begin{proof}
Let $(M^{\tilde{1}}, \vdash)$ be an interassociate of the semigroup $(M,\dashv)^{\tilde{1}}$. Then for any $x,y\in M$ we have the following equalities: $$x\vdash y=(x\dashv \tilde{1})\vdash(\tilde{1}\dashv y)=x\dashv(\tilde{1}\vdash \tilde{1})\dashv y= x\dashv a \dashv y,$$   where $a=\tilde{1}\vdash \tilde{1}\in M^{\tilde{1}}$.

Consider two cases.

(1) Let $a=\tilde{1}$. Then $x\vdash y=x\dashv\tilde{1}\dashv y=x\dashv y$ for all $x,y\in M$. Taking into account that $\tilde{1}\vdash \tilde{1}=\tilde{1}$ and for any $x\in M$  the following equalities hold:
$$x\vdash\tilde{1}=(x\dashv\tilde{1})\vdash\tilde{1}=x\dashv(\tilde{1}\vdash\tilde{1})=x\dashv\tilde{1}=x,$$
$$\tilde{1}\vdash x=\tilde{1}\vdash(\tilde{1}\dashv x)=(\tilde{1}\vdash\tilde{1})\dashv x=\tilde{1}\dashv x=x,$$
we conclude that in this case $(M^{\tilde{1}}, \vdash)=(M,\dashv)^{+1}$.

(2) Let $a\neq\tilde{1}$, and thus $a\in M$. We claim that $\tilde{1}\vdash x, x\vdash\tilde{1}\in M$ for any $x\in M^{\tilde{1}}$. Suppose that $\tilde{1}\vdash c=\tilde{1}$ for some $c\in M$. Then $e=\tilde{1}\dashv\tilde{1}=(\tilde{1}\vdash c)\dashv\tilde{1}=\tilde{1}\vdash(c\dashv\tilde{1})=\tilde{1}\vdash c=\tilde{1}$, and we have a contradiction. By analogy, $x\vdash\tilde{1}\in M$ for any $x\in M^{\tilde{1}}$.

For any $x\in M^{\tilde{1}}$ we have that
$$x\vdash\tilde{1}=(x\vdash\tilde{1})\dashv\tilde{1}=x\vdash(\tilde{1}\dashv\tilde{1})=x\vdash e,$$
$$\tilde{1}\vdash x=\tilde{1}\dashv(\tilde{1}\vdash x)=(\tilde{1}\dashv\tilde{1})\vdash x=e\vdash x.$$

Taking into account that for $a=\tilde{1}\vdash\tilde{1}\in M$
$$\tilde{1}\vdash\tilde{1}=\tilde{1}\dashv(\tilde{1}\vdash\tilde{1})\dashv\tilde{1}=\tilde{1}\dashv a\dashv\tilde{1}$$and for any $x\in M$
$$\tilde{1}\vdash x=e\vdash x=e\dashv a\dashv x=\tilde{1}\dashv a\dashv x\ \ \text{  and}$$
$$x\vdash\tilde{1}=x\vdash e=x\dashv a\dashv e=x\dashv a\dashv\tilde{1},$$
we conclude that $(M^{\tilde{1}}, \vdash)$ is a variant of $(M,\dashv)^{\tilde{1}}$ with the sandwich operation $x\vdash y=x\dashv a\dashv y$, where $a=\tilde{1}\vdash \tilde{1}\in M$.

Let $(M,\dashv)$ be a commutative monoid. 
Taking into account that for each $a\in M$ the variants with respect to $a$ of $(M,\dashv)^{\tilde{1}}$ and  $(M,\dashv)^{+1}$  coincide, and the set of interassociates of  $(M,\dashv)^{\tilde{1}}$ consists of  $(M,\dashv)^{+1}$ and variants of $(M,\dashv)^{\tilde{1}}$ with respect to all $a\in M$, we conclude that each interassociate of $(M,\dashv)^{\tilde{1}}$ is an interassociate of
$(M,\dashv)^{+{1}}$. Since $(M,\dashv)^{+{1}}$ is a monoid, all of its interassociates are variants. Consequently, $Int((M,\dashv)^{\tilde{1}})=Int((M,\dashv)^{+{1}})$. Let $(M^{\tilde{1}}, \vdash_1)$ and $(M^{\tilde{1}}, \vdash_2)$ be any two interassociate of $(M,\dashv)^{+{1}}$. Then $x \vdash_1 y=x\dashv a_1\dashv y$ and $x \vdash_2 y=x\dashv a_2\dashv y$ for some $a_1,a_2\in M^{\tilde{1}}$ and any $x,y\in M^{\tilde{1}}$.
Taking into account that $(M,\dashv)^{+{1}}$ is commutative and hence $x\vdash_1(y\vdash_2 z)=x\vdash_1(y\dashv a_2\dashv z)= x\dashv a_1 \dashv y\dashv a_2\dashv z= x\dashv a_2 \dashv y\dashv a_1\dashv z =x\vdash_2(y\dashv a_1\dashv z) =x\vdash_2(y\vdash_1 z)$, we conclude that $\vdash_1$ and $\vdash_2$ are strong interassociate.

\end{proof}

In the following Propositions~\ref{intLO+0} and~\ref{intRO+0} we recognize all (strong) interassociates of the semigroups $LO_X^{+0}$ and  $RO_X^{+0}$.

\begin{proposition}\label{intLO+0}
The set $Int(LO_X^{+0})$ consists of  all semigroups $LO^{\sim
0}_{A\leftarrow X}$, where $A\subset X$. Any two interassociates
of $LO_X^{+0}$ are interassociate with each other. The semigroup
$LO^{\sim 0}_{A\leftarrow X}$ is a strong interassociate of the
semigroup $LO^{\sim 0}_{B\leftarrow X}$ if and only if $A=B$ or
$A=\emptyset$ or $B=\emptyset$.
\end{proposition}

\begin{proof}
Let $(X^{0}, \vdash)$ be an interassociate of the semigroup
$LO_X^{+0}$ with operation $\dashv$.

If $a\vdash b=0$ for some $a,b\in X$ then
$$x\vdash b=(x\dashv a)\vdash b=x\dashv(a\vdash b)=x\dashv 0=0$$ for any $x\in X^0$.

If $c\vdash d\neq 0$ for some $c,d\in X$ then
$$x\vdash d=(x\dashv c)\vdash d=x\dashv(c\vdash d)=x$$ for any $x\in X^0$.

Let $A=\{a\in X\ |\ x\vdash a\neq 0\text{ for any }x\in X\}$.
It follows that $(X^{0}, \vdash)$ coincides with $LO^{\sim 0}_{A\leftarrow X}$.

\smallskip

Let us show that for any $A,B\subset X$ the semigroups $LO^{\sim 0}_{A\leftarrow X}$ with operation $\dashv_A$ and $LO^{\sim 0}_{B\leftarrow X}$ with operation $\vdash_B$ are interassociate with each other.

To prove $x\vdash_B(y\dashv_A z)=(x\vdash_B y)\dashv_A z$ consider the following two cases:

\begin{itemize}
    \item if $z\in A$ then $x\vdash_B(y\dashv_A z)=x\vdash_B y=(x\vdash_B y)\dashv_A z$ for any $x,y\in X^0$;
    \item if $z\in X^0\setminus A$ then $x\vdash_B(y\dashv_A z)=x\vdash_B 0=0=(x\vdash_B y)\dashv_A z$ for any $x,y\in X^0$.
\end{itemize}

To prove $x\dashv_A(y\vdash_B z)=(x\dashv_A y)\vdash_B z$  consider the following two cases:

\begin{itemize}
    \item if $z\in B$ then $x\dashv_A(y\vdash_B z)=x\dashv_A y=(x\dashv_A y)\vdash_B z$ for any $x,y\in X^0$;
    \item if $z\in X^0\setminus B$ then $x\dashv_A(y\vdash_B z)=x\dashv_A 0=0=(x\dashv_A y)\vdash_B z$ for any $x,y\in X^0$.
\end{itemize}

Let us prove that a semigroup $LO^{\sim 0}_{A\leftarrow X}$  is a strong interassociate of a semigroup $LO^{\sim 0}_{B\leftarrow X}$  if and only if $A=B$ or $A=\emptyset$ or $B=\emptyset$.

If $A=B$ then $LO^{\sim 0}_{A\leftarrow X}=LO^{\sim 0}_{B\leftarrow X}$. So, $LO^{\sim 0}_{A\leftarrow X}$ is a strong interassociate of a semigroup $LO^{\sim 0}_{B\leftarrow X}$.

If $A=\emptyset$ or $B=\emptyset$ then $LO^{\sim 0}_{A\leftarrow X}$ or $LO^{\sim 0}_{B\leftarrow X}$ is a null semigroup. Since a null semigroup is a strong interassociate of any semigroup with zero, in this case, $LO^{\sim 0}_{A\leftarrow X}$ and $LO^{\sim 0}_{B\leftarrow X}$ are strong interassociate with each other.

Let $A$ and $B$ are different non-empty subsets of $X$. Show that $LO^{\sim 0}_{A\leftarrow X}$ and $LO^{\sim 0}_{B\leftarrow X}$ are  not strong interassociate with each other.
For this, it is sufficient to consider the following two cases.
\begin{itemize}
    \item There are exist $a\in A$ and $b\in B\setminus A$. Then
    $a\vdash_B(a\dashv_A b)=a\vdash_B 0=0$ while $a\dashv_A(a\vdash_B b)=a\dashv_A a=a\neq 0$.
    \item There are exist $b\in B$ and $a\in A\setminus B$. Then
$b\dashv_A(b\vdash_B a)=b\dashv_A 0=0$ while $b\vdash_B(b\dashv_A a)=b\vdash_B b=b\neq 0$.
\end{itemize}

\end{proof}

 Taking into account that $(X,\dashv) $ is an interassociate of $(X,\vdash)$ if and only if
$(X,\dashv^d) $ is an interassociate of $(X,\vdash^d)$, and for each $A\subset X$ the semigroup $LO^{\sim 0}_{A\leftarrow X}$ is dual to $RO^{\sim 0}_{A\leftarrow X}$, we conclude the following proposition.

\begin{proposition}\label{intRO+0}
    The set $Int(RO_X^{+0})$ consists of  all semigroups $RO^{\sim 0}_{A\leftarrow X}$, where $A\subset X$. Any two interassociates of $RO_X^{+0}$ are interassociate with each other.
    The semigroup $RO^{\sim 0}_{A\leftarrow X}$ is a strong interassociate of the semigroup $RO^{\sim 0}_{B\leftarrow X}$ if and only if $A=B$ or $A=\emptyset$ or $B=\emptyset$.
\end{proposition}

Let  $a$ and $c$ be different elements of a set $X$. Consider the semigroup $LOB_X=(X, \dashv_c^a)$, where the binary operation $\dashv_c^a$ on  $X$ is defined in the following way:

$$x\dashv_c^a y=\begin{cases}
x,\text{ if } x\neq c \\
a,\text{ if } x=c\text{ and } y\neq c\\
c,\text{ if } x=y=c.
\end{cases}$$

\begin{proposition}\label{intROac}
If $(X,\vdash)$ is an interassociate of $(X,\dashv_c^a)$ then $(X,\vdash)=(X,\dashv_c^a)$
or $(X,\vdash)=LO_{X\setminus\{c\}\leftarrow X} =(X,\vdash^a_c)$, where
$$x\vdash^a_c y=\begin{cases}
x,\ x\neq c \\
a,\ x=c.
\end{cases}$$
All interassociates of $(X,\dashv_c^a)$ are strong.
\end{proposition}

\begin{proof} Since each element $z\in X\setminus\{c\}$ is a left zero of the semigroup $(X,\dashv_c^a)$, $z$ is a left zero of  $(X,\vdash)$.

For each $x\in X$ we have
$$c\vdash x=(c\dashv_c^a c)\vdash x=c\dashv_c^a(c\vdash x)\in\{a, c\}.$$

If $x\neq c$ then for each $y\in X$ the following equalities hold:
$$c\vdash x=c\vdash(x\dashv_c^a y)=(c\vdash x)\dashv_c^a y.$$
It follows that $c\vdash x$ is a left zero, and therefore, $c\vdash x\in X\setminus\{c\}$ for all $x\neq c$. Consequently, $c\vdash x=a$ for all $x\neq c$.

If $c\vdash c=c$ then $(X,\vdash)=(X,\dashv_c^a)$. If $c\vdash
c=a$ then $(X,\vdash)=(X,\vdash^a_c)$.

\smallskip

Let us show that $(X,\vdash^a_c)$ is a strong interassociate of
$(X,\dashv_c^a)$. Since each element $x\in X\setminus\{c\}$ is a
left zero of  $(X,\vdash^a_c)$ and $(X,\dashv_c^a)$,
$x\dashv_c^a (y\vdash^a_c z)=x=x\vdash^a_c (y\dashv_c^a z)$ for any
$x\in X\setminus\{c\}$ and $y, z\in X$. Taking into account that
$c\dashv_c^a (y\vdash^a_c z)\in c\dashv_c^a (X\setminus\{c\})=\{a\}$
and $c\vdash^a_c(y\dashv_c^a z)=a$, we conclude that
$c\dashv_c^a(y\vdash^a_c z)=c\vdash^a_c(y\dashv_c^a z)$ for any $y,
z\in X$.
\end{proof}

Dually one can  characterize all interassociates of the semigroup $ROB_X$.

\section{Three-element doppelsemigroups and their automorphism groups}\label{3eldsg}

I this section we describe up to isomorphism all (strong) doppelsemigroups with at most three elements and their automorphism groups.

Firstly, recall some useful facts  which we shall often use in
this section. In fact, each semigroup $(S,\dashv)$ can be consider
as a (strong) doppelsemigroup $(S,\dashv,\dashv)$ with the
automorphism group $\Aut(S,\dashv,\dashv)=\Aut(S,\dashv)$, and we
denote this {\em trivial} doppelsemigroup by $S$. As always, we denote
by $(S,\dashv_a)$ a variant of a semigroup $(S,\dashv)$, where
$x\dashv_a y= x\dashv a\dashv y$. If the semigroups $(S,\dashv_a)$
and $(S,\dashv_b)$ are variants of a commutative semigroup
$(S,\dashv)$ then the doppelsemigroup $(S,\dashv_a, \dashv_b)$ is
strong. If semigroup is a monoid then all of its interassociates
are variants. A semigroup coincides with each of its
interassociates if and only if it is a rectangular band, see
\cite[Lemma 5.5]{BoGN}. Every group is isomorphic to each of its
interassociates, see \cite{Dr}. Following the algebraic tradition, we take for a
model of the class of cyclic groups of order $n$ the multiplicative group
$C_n=\{z\in\mathbb C:z^n=1\}$ of $n$-th roots of $1$.

Let $(D_1,\dashv_1, \vdash_1)$ be such a doppelsemigroup that for
each doppelsemigroup $(D_2,\dashv_2, \vdash_2)$ the isomorphisms \break
$(D_2,\dashv_2)\cong (D_1,\dashv_1)$ and $(D_2,\vdash_2)\cong
(D_1,\vdash_1)$ imply $(D_2,\dashv_2, \vdash_2)\cong
(D_1,\dashv_1, \vdash_1)$. If $\mathbb S$ and $\mathbb T$ are
model semigroups of classes of semigroups isomorphic to
$(D_1,\dashv_1)$ and $(D_1, \vdash_1)$, respectively, then by
$\mathbb S\between\mathbb T$  we
denote a model doppelsemigroup of the class of doppelsemigroups
isomorphic to $(D_1,\dashv_1, \vdash_1)$.

Note that if $(D,\dashv, \vdash)$ is a (strong) doppelsemigroup then $(D,\vdash, \dashv)$ is a (strong) doppelsemigroup as well.
In general case, the doppelsemigroups $(D,\dashv, \vdash)$ and $(D,\vdash, \dashv)$ are not isomorphic.
It is clear that  $\Aut(D,\dashv, \vdash)=\Aut(D,\vdash, \dashv)$.

\bigskip

It is well-known that there are exactly five pairwise non-isomorphic semigroups having
two elements: $C_2$, $L_2$, $O_2$, $LO_2$, $RO_2$.

Consider the cyclic group $C_2=\{-1, 1\}$ and find up to isomorphism all doppelsemigroups
$(D,\dashv, \vdash)$ with \break $(D,\dashv)\cong C_2$. Because  $C_2$ is a monoid, all of its interassociates are variants. Since every group is isomorphic to each of its interassociates, in this case there are two (strong) doppelsemigroups up to isomorphism: $C_2$ and $C_2\between C_2^{-1}=(\{-1, 1\}, \cdot, \cdot_{-1})$. These doppelsemigroups are not isomorphic. Indeed, let $\psi$ is an isomorphism from
$(\{-1, 1\}, \cdot, \cdot_{-1})$ to $(\{-1, 1\}, \cdot, \cdot)$. Taking into account that $-1$ is a neutral element of the group $(\{-1, 1\},\cdot_{-1})$ and $\psi$ must preserve the neutral elements of both groups $(\{-1, 1\},\cdot)$ and $(\{-1, 1\},\cdot_{-1})$ of the doppelsemigroup
$(\{-1, 1\}, \cdot, \cdot_{-1})$, we conclude that $\psi(1)=1$ and $\psi(-1)=1$, which contradicts the assertion that $\psi$ is an isomorphism. Since $\Aut(C_2)\cong C_1$, $\Aut(C_2\between C_2^{-1})\cong C_1$.

Since $LO_2$ and  $RO_2$ are rectangular bands, all their interassociates coincide with them, and therefore, in this case there are only two doppelsemigroups: $LO_2$ and  $RO_2$.

It is well-known that a null semigroup $O_X$ is an interassociate
of each semigroup on $X$ with the same zero. Consequently, $O_2$ has two non-isomorphic interassociates: $O_2$ and $L_2$. Taking into account that the semilattice $L_2$ is the monoid $(\{0,1\},\min )$, we
conclude that $L_2$ has two non-isomorphic interassociates: $L_2$ and $(\{0,1\},\min_{0} )=O_2$. By
Propositions~\ref{nulldop} and~\ref{uniquedop}, it follows that
the last four non-isomorphic doppelsemigroups are  $O_2$, $O_2\between L_2$, $L_2$ and $L_2\between O_2$. Note that commutativity of $L_2$ implies that all these doppelsemigroups are strong. By
Proposition~\ref{autnulldop}, $\Aut(L_2\between O_2)=\Aut(O_2\between L_2)=\Aut(L_2)\cong C_1$.

Consequently, there exist  $6$  pairwise non-isomorphic commutative two-element doppelsemigroups  and $2$  non-isomorphic non-commutative doppelsemigroups of order $2$. All two-element  doppelsemigroups are strong.

In the following table we present  up to isomorphism all two-element  doppelsemigroups and their automorphism groups.

\begin{table}[ht]
    \centering

    \begin{tabular}{|c|cccccc|cc|}
        \hline
        $D$ & $C_2$  & $O_2$ & $L_2$  & $C_2\between C_2^{-1}$ & $O_2\between L_2$ & $L_2\between O_2$ & $LO_2$ & $RO_2$  \\
        \hline
        $\Aut(D)$ &  $C_1$ & $C_1$ & $C_1$ & $C_1$ & $C_1$ & $C_1$  & $C_2$ & $C_2$  \\
        \hline
    \end{tabular}

    \smallskip
    \caption{Two-element doppelsemigroups and their automorphism groups}\label{tab:DopSG2}
\end{table}

\smallskip

In the remaining part of the paper we concentrate on describing up to isomorphism all  three-element (strong)  doppelsemigroups.

Among $19683$ different binary operations on a three-element set
$S$ there are exactly $113$ operations which are
associative. In other words, there exist exactly
$113$ three-element semigroups, and many of these are isomorphic
so that there are essentially only $24$ pairwise non-isomorphic
semigroups of order $3$,  see \cite{Ch, G8, G9}.

Among $24$  pairwise non-isomorphic semigroups of order $3$ there are $12$  commutative
semigroups. The rest $12$ pairwise non-isomorphic non-commutative
three-element semigroups are divided into the pairs of dual semigroups that are antiisomorphic. The automorphism groups of dual semigroups coincide.

List of all pairwise non-isomorphic semigroups of order $3$ and their automorphism groups are presented in Table~\ref{tab:auts3} and Table~\ref{tab:autns3} taken from~\cite{G9}.

\begin{table}[ht]
    \centering

    \begin{tabular}{|c|cccccccccccc|}
        \hline
        $S$ & $C_3$ & $O_3$  & $\M_{2,2}$ & $C_2^{+1}$ &  $C_2^{\tilde{1}}$ & $\M_{3,1}$&  $O_2^{+1}$ & $O_2^{+0}$ & $L_3$ & $C_2^{+0}$    & $O_3^2$  & $O_3^1$ \\
        \hline
        $\Aut(S)$ &  $C_2$ & $C_2$ & $C_1$ & $C_1$ & $C_{1}$  & $C_{1}$ & $C_{1}$ & $C_{1}$ & $C_{1}$ & $C_{1}$ & $C_{2}$ & $C_{1}$  \\
        \hline
    \end{tabular}

    \smallskip
    \caption{Commutative semigroups $S$ of order  $3$ and their automorphism groups}\label{tab:auts3}
\end{table}

\begin{table}[ht]
    \centering
    \resizebox{14cm}{!}{
        \begin{tabular}{|c|c|c|c|c|c|c|}
            \hline
            $S$ & $LO_3$, $RO_3$  & $LO_2^{+0}$, $RO_2^{+0}$ &  $LO^{\sim 0}_{1\leftarrow2}$, $RO^{\sim 0}_{1\leftarrow2}$ & $LO_2^{+1}$, $RO_2^{+1}$ & $LOB_3$, $ROB_3$  & $LO_{2\leftarrow 3}$, $RO_{2\leftarrow 3}$ \\
            \hline
            $\Aut(S)$ & $S_3$ & $C_2$ & $C_1$ & $C_{2}$ & $C_{1}$ & $C_{2}$ \\
            \hline
        \end{tabular}
    }
    \smallskip
    \caption{Non-commutative three-element  semigroups and their automorphism groups}\label{tab:autns3}
\end{table}

In the sequel, we divide our investigation into cases. In the case
of a semigroup $S$ we shall find all doppelsemigroups $(D,\dashv,
\vdash)$ such that $(D,\dashv)$ is isomorphic to $S$.

\smallskip

{\noindent \bf Case $C_3$}.   Up to isomorphism, the multiplicative group $C_3=\{1, a, a^{-1}\}$, where $a=e^{2\pi i/3}$,
 is a unique group of order $3$. Since  $C_3$ is a monoid, all of its interassociates are variants.
 Because $C_3$ is commutative, all of its interassociates are strong.  Since every group is isomorphic to each of its
  interassociates, in this case there are exactly three (strong) doppelsemigroups: $C_3$, $(C_3, \cdot, \cdot_{a})$ and $(C_3, \cdot, \cdot_{a^{-1}})$.  It is easy to check the map $\psi : C_3\to C_3$, $\psi(g)=g^{-1}$ (where $g^{-1}$ is the inverse of $g$ in the group $(C_3, \cdot)$), is an isomorphism from   $(C_3, \cdot, \cdot_{a})$ to $(C_3, \cdot, \cdot_{a^{-1}})$.
We denote by $C_3\between C_3^{-1}$ the doppelsemigroup $(C_3, \cdot, \cdot_{a^{-1}})$. By the same arguments as for the group $C_2$, we conclude that the doppelsemigroups $C_3$ and $C_3\between C_3^{-1}$ are non-isomorphic. Let $\psi: C_3\to C_3$ is an automorphism of the
doppelsemigroup $(C_3, \cdot, \cdot_{a^{-1}})$. Taking into account
that $1$ is the identity of the group $(C_3, \cdot)$ and $a$ is
the identity of the group $(C_3, \cdot_{a^{-1}})$, we conclude that
$\psi(1)=1$ and $\psi(a)=a$. Consequently, $\psi(a^{-1})=a^{-1}$,
and $\psi$ is the identity automorphism. It follows that
$\Aut(C_3\between C_3^{-1})\cong C_1$. 

\smallskip

{\noindent \bf Case $O_3$}. A null semigroup $O_3$ is a (strong) interassociate of each three-element semigroup with the same zero. Thus, up to isomorphism there are the following 12 (strong) doppelsemigroups: $O_3$, $O_3\between \M_{3,1}$, $O_3\between O_2^{+1}$, $O_3\between O_2^{+0}$, $O_3\between L_3$,  $O_3\between C_2^{+0}$, $O_3\between O_3^2$, $O_3\between O_3^1$,   $O_3\between LO_2^{+0}$, $O_3\between RO_2^{+0}$,
$O_3\between LO^{\sim 0}_{1\leftarrow2}$, $O_3\between RO^{\sim 0}_{1\leftarrow2}$.
According to Proposition~\ref{nulldop}, up to isomorphism there are no other doppelsemigroups $(D,\dashv, \vdash)$ such that $(D,\dashv)\cong O_3$. By Proposition~\ref{autnulldop}, $\Aut(O_3\between S)\cong \Aut(S)$
for any three-element semigroup $S$ with zero.

\smallskip

{\noindent \bf Case $\M_{2,2}$}. Consider the monogenic semigroup
$\M_{2,2}=\{a,a^2,a^3\ |\ a^4=a^2\}$. There are three
interassociates of this semigroup: $(\M_{2,2}, *_k)$, where
$a^x*_ka^y=a^{x+y+k-2}$ for every $a^x, a^y\in \M_{2,2}$ and
$k\in\{1,2,3\}$, see~\cite[Theorem 1.1]{GLN}. It easy to check
that $(\M_{2,2}, *_1)=\{a^2, a^3\}^{+1}\cong C_2^{+1}$,
$(\M_{2,2}, *_2)=(\M_{2,2}, *)$ and $(\M_{2,2}, *_3)=\{a^2,
a^3\}^{\tilde{1}}\cong C_2^{\tilde{1}}$. So, in this case there
are three  doppelsemigroups: $\M_{2,2}$, $\M_{2,2}\between
C_2^{+1}$ and $\M_{2,2}\between C_2^{\tilde{1}}$. Since all three
interassociates of $\M_{2,2}$  are pairwise non-isomorphic,
according to Proposition~\ref{uniquedop} we conclude that up to
isomorphism there are no other doppelsemigroups $(D,\dashv,
\vdash)$ such that $(D,\dashv)\cong \M_{2,2}$. Since
$\Aut(\M_{2,2})\cong C_1$, $\Aut(\M_{2,2}\between C_2^{+1})\cong C_1$ and $\Aut(\M_{2,2}\between C_2^{\tilde{1}})\cong C_1$.

\smallskip

{\noindent \bf Case $C_2^{+1}$}. Since $C_2^{+1}$ is a monoid, all of its interassociates are variants. Let $e$ be an extra identity adjoined to $C_2=\{-1,1\}$. Then $(\{-1,1, e\}, \cdot_e)= C_2^{+1}$,
$(\{-1,1, e\}, \cdot_1)\cong C_2^{\tilde{1}}$ and $(\{-1,1, e\}, \cdot_{-1})\cong \M_{2,2}$. Therefore, there are three  doppelsemigroups: $C_2^{+1}$,  $C_2^{+1}\between C_2^{\tilde{1}}$ and $C_2^{+1}\between\M_{2,2}$. Since $C_2^{+1}$ is a commutative monoid, all these doppelsemigroups are strong. Taking into account that all three interassociates of $C_2^{+1}$  are pairwise non-isomorphic, by Proposition~\ref{uniquedop} we conclude that up to isomorphism there are no other doppelsemigroups $(D,\dashv, \vdash)$ such that $(D,\dashv)\cong C_2^{+1}$.
Since $\Aut(C_2^{+1})\cong C_1$, $\Aut(C_2^{+1}\between C_2^{\tilde{1}}) \cong C_1$ and $\Aut(C_2^{+1}\between\M_{2,2}) \cong C_1$.

\smallskip

{\noindent \bf Case $C_2^{\tilde{1}}$}. According to
Proposition~\ref{intmonoid+almost1} the semigroup
$C_2^{\tilde{1}}$ has three interassociates. As we have seen in
previous cases, these interassociates must be isomorphic to
$C_2^{\tilde{1}}$, $C_2^{+1}$ and $\M_{2,2}$. Taking into account
that, by Proposition~\ref{intmonoid+almost1}, 
all interassociates of $C_2^{\tilde{1}}$ are strong, we conclude that
in this case there are three pairwise non-isomorphic strong doppelsemigroups:
$C_2^{\tilde{1}}$, $C_2^{\tilde{1}}\between C_2^{+1}$ and $C_2^{\tilde{1}}\between
\M_{2,2}$.  Since $\Aut(C_2^{\tilde{1}})\cong C_1 $, $\Aut(C_2^{\tilde{1}}\between C_2^{+1})\cong C_1$
and $\Aut(C_2^{\tilde{1}}\between \M_{2,2})\cong C_1$.

\smallskip

{\noindent \bf Case $\M_{3,1}$}. Consider the monogenic semigroup
$\M_{3,1}=\{a,a^2,a^3\ |\ a^4=a^3\}$. There are three
interassociates of this semigroup: $(\M_{3,1}, *_k)$, where
$a^x*_ka^y=a^{x+y+k-2}$ for every $a^x, a^y\in \M_{3,1}$ and
$k\in\{1,2,3\}$, see~\cite[Theorem 1.1]{GLN}. It easy to check
that $(\M_{3,1}, *_1)=\{a^2, a^3\}^{+1}\cong O_2^{+1}$,
$(\M_{3,1}, *_2)=(\M_{3,1}, *)$ and $(\M_{3,1}, *_3)\cong O_3$.
So, in this case we have three  doppelsemigroups: $\M_{3,1}$,
$\M_{3,1}\between O_2^{+1}$ and $\M_{3,1}\between O_3$. Since all three interassociates of
$\M_{3,1}$  are pairwise non-isomorphic, according to
Proposition~\ref{uniquedop} we conclude that up to isomorphism
there are no other doppelsemigroups $(D,\dashv, \vdash)$ such that
$(D,\dashv)\cong \M_{3,1}$. Since $\Aut(\M_{3,1})\cong C_1$,
$\Aut(\M_{3,1}\between O_2^{+1})\cong C_1$ and $\Aut(\M_{3,1}\between O_3)\cong C_1$.

\smallskip

{\noindent \bf Case $O_2^{+1}$}. Because the semigroup  $O_2^{+1}$ is a monoid, all of its interassociates are variants. Since there are only three variants of a three-element semigroup,  previous cases imply that these interassociates isomorphic to $O_2^{+1}$, $\M_{3,1}$ and $O_3$.
Taking into account that $O_2^{+1}$ is a commutative monoid, we conclude that in this case there are three pairwise non-isomorphic strong doppelsemigroups: $O_2^{+1}$, $O_2^{+1}\between\M_{3,1}$ and $O_2^{+1}\between O_3$.  Since $\Aut(O_2^{+1})\cong C_1$, $\Aut(O_2^{+1}\between\M_{3,1})\cong C_1$ and $\Aut(O_2^{+1}\between O_3)\cong C_1$.

\smallskip

{\noindent \bf Case $O_2^{+0}$}. Proposition~\ref{intnull+0}
implies that there are three interassociates of semigroup
$O_2^{+0}$, and all these interassociates are strong. They are
isomorphic to $O_2^{+0}$, $L_3$ and $O_3$. So, in this case there
are three strong doppelsemigroups $O_2^{+0}$, $O_2^{+0}\between
L_3\cong (O_2\between L_2)^{+0}$ and $O_2^{+0}\between O_3$. Since all three interassociates
of $O_2^{+0}$  are pairwise non-isomorphic, according to
Proposition~\ref{uniquedop} we conclude that up to isomorphism
there are no other doppelsemigroups $(D,\dashv, \vdash)$ such that
$(D,\dashv)\cong O_2^{+0}$. Since $\Aut(O_2^{+0})\cong C_1$,
$\Aut(O_2^{+0}\between L_3)\cong C_1$ and $\Aut(O_2^{+0}\between O_3)\cong C_1$.

\smallskip

{\noindent \bf Case $L_3$}. Since the linear semilattice  $L_3$ is a monoid, all of its interassociates are variants. Three-element semigroup has only three variants, thus  previous cases imply that these interassociates isomorphic to $L_3$, $O_2^{+0}$ and $O_3$.
Therefore, in this case we have the trivial doppelsemigroup $L_3$ and two (strong) doppelsemigroups $L_3 \between O_2^{+0}\cong (L_2 \between O_2)^{+0}$ and $L_3 \between O_3$.  Since $\Aut(L_3)\cong C_1$,
$\Aut(L_3 \between O_2^{+0})\cong C_1$ and $\Aut(L_3 \between O_3)\cong C_1$.

\smallskip

{\noindent \bf Case $C_2^{+0}$}. Consider the semigroup $C_2^{+0}$ isomorphic to a commutative monoid $(\{-1,1,0\}, \cdot)$ with  zero $0$. Except a null semigroup $O_3$, this monoid has two isomorphic variants
$(\{-1,1,0\}, \cdot)$ and $(\{-1,1,0\}, \cdot_{-1})$.
In this case there are three (strong) doppelsemigroups: $C_2^{+0}\between O_3$, $C_2^{+0}$ and $(\{-1,1,0\}, \cdot, \cdot_{-1})$. These doppelsemigroups are not isomorphic. Indeed, let $\psi$ is an isomorphism from
$(\{-1,1,0\}, \cdot, \cdot_{-1})$ to $(\{-1,1,0\}, \cdot, \cdot)$. Taking into account that $-1$ is a neutral element of the semigroup $(\{-1,1,0\},\cdot_{-1})$ and $\psi$ must preserve the neutral elements of both semigroups $(\{-1,1,0\},\cdot)$ and $(\{-1,1,0\},\cdot_{-1})$ of the doppelsemigroup
$(\{-1,1,0\}, \cdot, \cdot_{-1})$, we conclude that $\psi(-1)=1$ and $\psi(1)=1$, which contradicts the assertion that $\psi$ is an isomorphism. Taking into account that $(\{-1,1,0\}, \cdot_{-1})\cong (C_2^{-1})^{+0}$, where $C_2^{-1}=(\{-1,1\}, \cdot_{-1})$, we denote by $C_2^{+0}\between (C_{2}^{-1})^{+0}$ the doppelsemigroup $(\{-1,1,0\}, \cdot, \cdot_{-1})$. It is easy to see that $C_2^{+0}\between (C_{2}^{-1})^{+0}\cong (C_2\between C_{2}^{-1})^{+0}$, and hence 
$\Aut(C_2^{+0}\between (C_{2}^{-1})^{+0})\cong \Aut((C_2\between C_{2}^{-1})^{+0})\cong \Aut(C_2\between C_{2}^{-1})\cong C_1$.
Since $\Aut(C_2^{+0})\cong C_1$,  $\Aut(C_2^{+0}\between O_3)\cong C_1$.

\smallskip

{\noindent \bf Case $O_3^2$}. Consider the non-linear semilattice $O_3^2$ isomorphic to the semigroup $\{a,b, 0\}$ with the operation $\dashv$:
$$x\dashv y=\begin{cases}
x,\text{ if } y=x\in\{a,b\} \\
0,\ \text{otherwise.}
\end{cases}$$

According to Proposition~\ref{intnullA}, this semigroup has four (strong) interassociates: $O_3^2$,
$O_3$, $(\{a,b, 0\},\vdash_a)$ and $(\{a,b, 0\},\vdash_b)$, where for $i\in\{a,b\}$
$$x\vdash_i y=\begin{cases}
x,\text{ if } y=x=i \\
0,\ \text{otherwise.}
\end{cases}$$
It is easy to check that the map $\psi : \{a,b, 0\}\to\{a,b, 0\}$, $\psi(a)=b$, $\psi(b)=a$ and $\psi(0)=0$, is a doppelsemigroup isomorphism from $(\{a, b, 0\}, \dashv,\vdash_a)$ to $(\{a,b, 0\}, \dashv,\vdash_b)$.

Therefore, in this case there are three pairwise non-isomorphic (strong) doppelsemigroups: $O_3^2$,
$O_3^2\between O_3$ and $O_3^2\between O_3^1$. Since $\Aut(O_3^1)\cong C_1$, $\Aut(O_3^2\between O_3^1)\cong C_1$. By Proposition~\ref{autnulldop},
$\Aut(O_3^2\between O_3)\cong \Aut(O_3\between O_3^2)\cong\Aut(O_3^2)\cong C_2$.

\smallskip

{\noindent \bf Case $O_3^1$}. Consider the last commutative semigroup $O_3^1$ isomorphic to the semigroup
$(\{a,b, 0\},\vdash_a)$ from the previous case. By Proposition~\ref{intnullA}, this semigroup has the same four (strong) interassociates as $O_3^2$. Show that the doppelsemigroups
$(\{a,b, 0\},\vdash_a, \vdash_a)$ and $(\{a,b, 0\},\vdash_a, \vdash_b)$ are not isomorphic. Suppose that $\psi$ is an isomorphism from $(\{a,b, 0\},\vdash_a, \vdash_b)$ to
$(\{a,b, 0\},\vdash_a, \vdash_a)$. Then $\psi$ must preserve a unique non-zero idempotent of these doppelsemigroups. Therefore, $\psi(a)=a$ and $\psi(b)=a$, which contradicts the assertion that $\psi$ is an isomorphism. Denote by $O_3^a\between O_3^b$ the doppelsemigroup $(\{a,b, 0\},\vdash_a, \vdash_b)$.  Thus, in this case we have four non-isomorphic (strong) doppelsemigroups: $O_3^1$, $O_3^a\between O_3^b$, $O_3^1\between O_3^2$ and $O_3^1\between O_3$. Since $\Aut(O_3^a)\cong\Aut(O_3^1)\cong C_1$,  $\Aut(O_3^a\between O_3^b)\cong C_1$, $\Aut(O_3^1\between O_3^2)\cong C_1$ and $\Aut(O_3^1\between O_3)\cong C_1$.

\bigskip

Let $(D,\dashv, \vdash)$ be a doppelsemigroup. Denote by
$(D,\dashv, \vdash)^d$ its {\em dual} doppelsemigroup
$(D,\dashv^d, \vdash^d)$, where $x\dashv^d y=y\dashv x$ and
$x\vdash^d y=y\vdash x$. In fact, $(D,\dashv, \vdash)^d$ is a
(strong) doppelsemigroup if and only if $(D,\dashv, \vdash)$ is a
(strong) doppelsemigroup. So, non-commutative doppelsemigroups are
divided into the pairs of dual doppelsemigroups. A map $\psi :
D_1\to D_2$ is a isomorphism from a doppelsemigroup
$(D_1,\dashv_1, \vdash_1)$ to $(D_2\dashv_2, \vdash_2)$ if and
only if  $\psi$ is a isomorphism from  a doppelsemigroup
$(D_1,\dashv_1, \vdash_1)^d$ to $(D_2\dashv_2, \vdash_2)^d$. Thus,
$\Aut((D,\dashv, \vdash)^d)=\Aut(D,\dashv, \vdash)$.

It follows that it is sufficient to consider non-commutative three-element  semigroups
$LO_3$, $LO_2^{+0}$,  $LO^{\sim 0}_{1\leftarrow2}$, $LO_2^{+1}$,  $LOB_3$, $LO_{2\leftarrow 3}$.
The cases of semigroups $RO_3$, $RO_2^{+0}$,  $RO^{\sim 0}_{1\leftarrow2}$, $RO_2^{+1}$,  $ROB_3$, $RO_{2\leftarrow 3}$ we shall get using the duality.

\smallskip

\smallskip

{\noindent \bf Case $LO_3$}. Since $LO_3$ is a rectangular band, all its interassociates coincide with $LO_3$, and therefore, in this case there is a unique doppelsemigroup $LO_3$.

\smallskip

{\noindent \bf Case $LO_2^{+0}$}. Consider the semigroup $LO_2^{+0}$ isomorphic to  $\{a,b, 0\}$ with the operation $\dashv$:
$$x\dashv y=\begin{cases}
x,\text{ if } y\in\{a,b\} \\
0,\text{ if } y=0.
\end{cases}$$

According to Proposition~\ref{intLO+0}, this semigroup has four interassociates: $LO_2^{+0}$, $O_3$, $(\{a,b, 0\},\vdash_a)$ and $(\{a,b, 0\},\vdash_b)$, where for $i\in\{a,b\}$
$$x\vdash_i y=\begin{cases}
x,\text{ if } y=i \\
0,\text{ if } y\neq i.
\end{cases}$$
It is easy to check that the map $\psi : \{a,b, 0\}\to\{a,b, 0\}$, $\psi(a)=b$, $\psi(b)=a$ and $\psi(0)=0$, is a doppelsemigroup isomorphism from $(\{a, b, 0\}, \dashv,\vdash_a)$ to $(\{a,b, 0\}, \dashv,\vdash_b)$. Since $(\{a, b, 0\}, \vdash_a)\cong(\{a,b, 0\}, \vdash_b)\cong LO^{\sim 0}_{1\leftarrow 2}$, denote by $LO_2^{+0}\between LO^{\sim 0}_{1\leftarrow 2}$ the doppelsemigroup $(\{a, b, 0\}, \dashv,\vdash_a)\cong(\{a,b, 0\}, \dashv,\vdash_b)$.

Thus, in this case we have three pairwise non-isomorphic doppelsemigroups: $LO_2^{+0}$, $LO_2^{+0}\between O_3$ and $LO_2^{+0}\between LO^{\sim 0}_{1\leftarrow 2}$. Consequently, up to isomorphism there are no other doppelsemigroups $(D,\dashv, \vdash)$ such that $(D,\dashv)\cong LO_2^{+0}$.
By Proposition~\ref{intLO+0}, the doppelsemigroups $LO_2^{+0}$ and $LO_2^{+0}\between O_3$ are strong while $LO_2^{+0}\between LO^{\sim 0}_{1\leftarrow 2}$ is not strong.

According to Proposition~\ref{autLZdop}, $\Aut(LO_2^{+0}\between LO^{\sim 0}_{1\leftarrow 2})\cong\Aut(LO^{\sim 0}_{1\leftarrow 2})\cong C_1$. 

By Proposition~\ref{autnulldop},  $\Aut(LO_2^{+0}\between O_3)\cong\Aut(O_3 \between LO_2^{+0})\cong\Aut(LO_2^{+0})\cong  C_2$.

\smallskip

{\noindent \bf Case $LO^{\sim 0}_{1\leftarrow 2}$}. Consider the
semigroup $LO^{\sim 0}_{1\leftarrow 2}$ isomorphic to the
semigroup $(\{a,b, 0\},\vdash_a)$ from the previous case. Since
this semigroup is the last semigroup with  zero, 
the previous cases imply that it has the following interassociates:
$O_3$, $LO_2^{+0}$, and interassociates that isomorphic to
$(\{a,b, 0\},\vdash_a)$.

Consider  interassociates of $(\{a,b, 0\},\vdash_a)$ that isomorphic to
$(\{a,b, 0\},\vdash_a)$. Since an isomorphism $\psi$ must preserve
a unique right identity $a$ and  zero $0$, we conclude that
$\psi(a)$ must be a right identity and $\psi(0)=0$. Thus, $(\{a,b,
0\},\vdash_b)$ is a unique different from $(\{a,b, 0\},\vdash_a)$
interassociate isomorphic to $(\{a,b, 0\},\vdash_a)$.
 Show that the doppelsemigroups
$(\{a,b, 0\},\vdash_a, \vdash_a)$ and $(\{a,b, 0\},\vdash_a,
\vdash_b)$ are not isomorphic. Suppose that $\psi$ is an
isomorphism from $(\{a,b, 0\},\vdash_a, \vdash_b)$ to $(\{a,b,
0\},\vdash_a, \vdash_a)$. Then $\psi$ must preserve right
identities $a$ and $b$ of  the semigroups $(\{a,b, 0\},\vdash_a)$
and $(\{a,b, 0\}, \vdash_b)$, respectively. Therefore, $\psi(a)=a$
and $\psi(b)=a$, which contradicts the assertion that $\psi$ is an
isomorphism. Denote by $LO^{\sim 0}_{a\leftarrow 2}\between
LO^{\sim 0}_{b\leftarrow 2}$ the doppelsemigroup $(\{a,b,
0\},\vdash_a, \vdash_b)$.  Thus, up to isomorphism, $LO^{\sim 0}_{1\leftarrow
2}$, $LO^{\sim 0}_{1\leftarrow 2}\between O_3$, $LO^{\sim 0}_{1\leftarrow 2}\between LO_2^{+0}$ and $LO^{\sim 0}_{a\leftarrow 2}\between LO^{\sim
0}_{b\leftarrow 2}$ are the last four doppelsemigroups with zero.
Since $\Aut(LO^{\sim 0}_{a\leftarrow
2})\cong\Aut(LO^{\sim 0}_{1\leftarrow 2})\cong C_1$,
 $\Aut(LO^{\sim 0}_{1\leftarrow 2}\between LO_2^{+0})\cong C_1$, $\Aut(LO^{\sim 0}_{1\leftarrow 2}\between O_3)\cong C_1$ and $\Aut(LO^{\sim 0}_{a\leftarrow 2}\between LO^{\sim 0}_{b\leftarrow
2})\cong C_1$.
By Proposition~\ref{intLO+0}, the doppelsemigroups $LO^{\sim
0}_{1\leftarrow 2}$ and $LO^{\sim 0}_{1\leftarrow 2}\between O_3$ are strong while $LO^{\sim 0}_{1\leftarrow 2}\between LO_2^{+0}$ and  $LO^{\sim 0}_{a\leftarrow 2}\between LO^{\sim 0}_{b\leftarrow 2}$ are not strong.

\smallskip

{\noindent \bf Case $LO_2^{+1}$}. Consider a monoid $LO_2^{+1}$ with operation $\dashv$ and identity $1$, where $LO_2=\{a, b\}$ is a two-element left zero semigroup. Since each interassociate of $LO_2^{+1}$ is a variant, we conclude that except $LO_2^{+1}$ there two  interassociates:  $(\{a, b, 1\}, \vdash_a)$ and
$(\{a, b, 1\}, \vdash_b)$ isomorphic to $LO_{2\leftarrow 3}$,
where for $i\in\{a,b\}$
$$x\vdash_i y=\begin{cases}
x,\ x\neq 1 \\
i,\ x=1.
\end{cases}$$
It is easy to check that the map $\psi : \{a,b, 1\}\to\{a,b, 1\}$, $\psi(a)=b$, $\psi(b)=a$ and $\psi(0)=0$, is a doppelsemigroup isomorphism from $(\{a, b, 1\}, \dashv,\vdash_a)$ to $(\{a,b, 1\}, \dashv,\vdash_b)$.
Since $1\dashv(b\vdash_a b)=1\dashv b=b$ while $1\vdash_a(b\dashv b)=1\vdash_a b=a\neq b$, the doppelsemigroup  $(\{a, b, 1\}, \dashv,\vdash_a)$ is not strong.
 Therefore, in this case there are two pairwise non-isomorphic  doppelsemigroups: $LO_2^{+1}$ and $LO_2^{+1}\between LO_{2\leftarrow 3}$. The semigroup
 $LO_2^{+1}$ is strong while $LO_2^{+1}\between LO_{2\leftarrow 3}$ is not strong.
By Proposition \ref{autLZdop},  $\Aut(LO_2^{+1}\between LO_{2\leftarrow 3})\cong  \Aut(LO_{2\leftarrow 3})\cong C_2$.

\smallskip

{\noindent \bf Case $LOB_3$}. Consider a non-commutative band
$LOB_3$ isomorphic to the semigroup $\{a,b,c \}$ with the
operation $\dashv_c^a$, where

$$x\dashv_c^a y=\begin{cases}
x,\text{ if } x\neq c \\
a,\text{ if } x=c\text{ and } y\neq c\\
c,\text{ if } x=y=c.
\end{cases}$$
By Proposition~\ref{intROac}, $LOB_3$ has two  interassociates isomorphic to $LOB_3$ and $LO_{2\leftarrow 3}$. According to Proposition~\ref{nulldop}, up to isomorphism there are no other doppelsemigroups $(D,\dashv, \vdash)$ such that $(D,\dashv)\cong LOB_3$. Thus, in this case there are two non-isomorphic doppelsemigroups: $LOB_3$ and $LOB_3\between LO_{2\leftarrow 3}$. By Proposition~\ref{intROac}, these doppelsemigroups are strong. Since $\Aut(LOB_3)\cong C_1$, $\Aut(LOB_3\between LO_{2\leftarrow 3})\cong C_1$.

\smallskip

{\noindent \bf Case $LO_{2\leftarrow 3}$}. Finally, consider the
last three-element semigroup $LO_{2\leftarrow 3}$ isomorphic to
the semigroup $\{a,b,c\}$ with operation $\dashv$ defined as
follows:
$$x\dashv y=\begin{cases}
x,\ x\neq c \\
a,\ x=c.
\end{cases}$$

Since this semigroup is the last semigroup, 
the previous cases imply that it has the following interassociates:
$LO_2^{+1}$, $LOB_3$, and interassociates that isomorphic to $(\{a,b,c\},\dashv)$.
Consider   interassociates of $(\{a,b,c\},\dashv)$ that isomorphic to
$(\{a,b,c\},\dashv)$. Since $a$ and $b$ are left zeros of
$(\{a,b,c\},\dashv)$, they must be left zeros of each
interassociate of $(\{a,b,c\},\dashv)$. It is clear that there
exists only one different from $(\{a,b,c\},\dashv)$ its
interassociate $(\{a,b,c\},\vdash)\cong(\{a,b,c\},\dashv)$, where
$$x\vdash y=\begin{cases}
x,\ x\neq c \\
b,\ x=c.
\end{cases}$$

It is easy to check that the map $\psi : \{a,b, c\}\to\{a,b, c\}$,
$\psi(a)=b$, $\psi(b)=a$ and $\psi(c)=c$, is a doppelsemigroup
isomorphism from $(\{a, b, c\}, \dashv,\dashv)$ to $(\{a,b, c\},
\dashv,\vdash)$. Consequently,  $LO_{2\leftarrow 3}$, $LO_{2\leftarrow 3}\between LO_2^{+1}$ and $LO_{2\leftarrow 3}\between LOB_3$ are the last three doppelsemigroups of order $3$.

It follows that $\Aut(LO_{2\leftarrow 3}\between LO_2^{+1})\cong \Aut(LO_2^{+1}\between LO_{2\leftarrow 3})\cong C_2$ and 
$\Aut(LO_{2\leftarrow 3}\between LOB_3)\cong \Aut(LOB_3\between LO_{2\leftarrow 3} )\cong C_1$.
Since $LOB_3\between LO_{2\leftarrow 3}$ is strong, $LO_{2\leftarrow 3}\between LOB_3$ is strong as well.
By analogy, $LO_{2\leftarrow 3}\between LO_2^{+1}$ is not strong.

\bigskip

We summarize the obtained results on the pairwise non-isomorphic non-trivial three-element (strong) doppelsemigroups and their automorphism groups in the following Tables~\ref{tab:autntcs3}, \ref{tab:autnts3} and \ref{tab:autntns3}.

\begin{table}[ht]
    \centering
    \resizebox{13cm}{!}{
        \begin{tabular}{|c|c|c|c|c|c|c|}
             \hline
            $D$ & $C_3\between C_3^{-1}$  &  $O_3\between \M_{3,1}$ & $O_3\between O_2^{+1}$ & $O_3\between O_2^{+0}$  & $O_3\between L_3$ & $O_3\between C_2^{+0}$  \\
              \hline
            $\Aut(D)$ & $C_1$  & $C_1$ & $C_{1}$ & $C_{1}$ & $C_{1}$ & $C_1$ \\
            \hline
            \hline
            $D$   & $O_3\between O_3^2$  & $O_3\between O_3^1$ &   $\M_{2,2}\between C_2^{+1}$  &  $\M_{2,2}\between C_2^{\tilde{1}}$ & $C_2^{+1}\between C_2^{\tilde{1}}$ & $C_2^{+1}\between \M_{2,2}$ \\
            \hline
            $\Aut(D)$  & $C_2$ & $C_1$ & $C_{1}$ & $C_{1}$ & $C_{1}$ & $C_{1}$\\
            \hline
            \hline
            $D$ & $C_2^{\tilde{1}}\between \M_{2,2} $ & $C_2^{\tilde{1}}\between C_2^{+1} $ & $\M_{3,1}\between O_2^{+1}$ & $\M_{3,1}\between O_3$  & $O_2^{+1}\between \M_{3,1}$ & $O_2^{+1}\between O_3$  \\
\hline
 $\Aut(D)$ & $C_1$ & $C_1$ & $C_1$ & $C_{1}$ & $C_{1}$ & $C_{1}$\\
 \hline
\hline
 $D$ & $(O_2\between L_2)^{+0}$  &  $ O_2^{+0}\between O_3$ & $L_3 \between O_3$ & 
   $(L_2\between O_2)^{+0}$ & $(C_2\between C_2^{-1})^{+0}$ & $C_2^{+0}\between O_3$     \\
             \hline
            $\Aut(D)$ & $C_1$ & $C_1$ & $C_1$ & $C_{1}$ & $C_{1}$ & $C_{1}$\\
            \hline
            \hline
            $D$ & $O_3^2\between O_3^1$ & $O_3^2\between O_3$  & $O_3^a\between O_3^b$ & $O_3^1\between O_3^2$ &  $O_3^1\between O_3$   \\
             \cline{1-6}
            $\Aut(D)$ & $C_1$ & $C_2$ & $C_1$ & $C_{1}$ & $C_{1}$  \\
            \cline{1-6}
        \end{tabular}

    }
    \smallskip
    \caption{  Three-element (strong) non-trivial commutative doppelsemigroups and their automorphism groups}\label{tab:autntcs3}
\end{table}

\begin{table}[ht]
    \centering
    \resizebox{16cm}{!}{
        \begin{tabular}{|c|c|c|c|c|c|c|}
            \hline
            $D$ & $O_3\between LO_2^{+0}$  & $O_3\between LO^{\sim 0}_{1\leftarrow2}$  & $LO_2^{+0}\between O_3$ &  $ LO^{\sim 0}_{1\leftarrow2}\between O_3$ & $LOB_3\between LO_{2\leftarrow 3}$  & $LO_{2\leftarrow 3}\between LOB_3$ \\
            & $O_3\between RO_2^{+0}$  & $O_3\between RO^{\sim 0}_{1\leftarrow2}$  & $RO_2^{+0}\between O_3$ & $ RO^{\sim 0}_{1\leftarrow2}\between O_3$ & $ROB_3\between RO_{2\leftarrow 3}$ & $RO_{2\leftarrow 3}\between ROB_3$    \\
             \hline
            $\Aut(D)$ & $C_2$ & $C_1$  & $C_{2}$ & $C_{1}$ & $C_{1}$ & $C_{1}$  \\
            \hline
              \end{tabular}
    }
    \smallskip
    \caption{Three-element non-trivial non-commutative strong doppelsemigroups and their automorphism groups}\label{tab:autnts3}
\end{table}

\begin{table}[ht]
    \centering
    \resizebox{16cm}{!}{
        \begin{tabular}{|c|c|c|c|c|c|}
           
             \hline
            $D$ & $LO_2^{+0}\between LO^{\sim 0}_{1\leftarrow 2}$ & $ LO^{\sim 0}_{1\leftarrow 2}\between LO_2^{+0}$ & $LO^{\sim 0}_{a\leftarrow 2}\between LO^{\sim 0}_{b\leftarrow 2}$ & $LO_2^{+1}\between LO_{2\leftarrow 3}$  &  $LO_{2\leftarrow 3}\between LO_2^{+1}$   \\
& $RO_2^{+0}\between RO^{\sim 0}_{1\leftarrow 2}$  & $ RO^{\sim 0}_{1\leftarrow 2}\between RO_2^{+0}$  & $RO^{\sim 0}_{a\leftarrow 2}\between RO^{\sim 0}_{b\leftarrow 2}$ & $RO_2^{+1}\between RO_{2\leftarrow 3}$  &  $RO_{2\leftarrow 3}\between RO_2^{+1}$    \\
            \hline
            $\Aut(D)$  & $C_1$ & $C_1$ & $C_1$ & $C_2$ & $C_{2}$   \\
           \hline
        \end{tabular}
    }
    \smallskip
    \caption{Three-element (non-commutative) non-strong doppelsemigroups and their automorphism groups}\label{tab:autntns3}
\end{table}

\newpage

It follows that we have proved the following theorem.

\begin{theorem}
There exist $75$ pairwise non-isomorphic three-element
doppelsemigroups among which $41$ doppelsemigroups are
commutative. Non-commutative doppelsemigroups are divided into
$17$ pairs of dual doppelsemigroups. Also up to isomorphism there
are $65$ strong doppelsemigroups of order $3$, and all non-strong
doppelsemigroups are not commutative.    There exist exactly $24$
pairwise non-isomorphic three-element trivial doppelsemigroups.
\end{theorem}

\section{Acknowledgment}

The authors would like to express their sincere thanks to  the anonymous referee
for a very careful reading of the paper and for all its insightful comments and valuable suggestions, which improve considerably the presentation of this paper.

\end{document}